\documentclass{amsart}

\usepackage{amsmath,amssymb,enumerate,amsthm}

\usepackage[llbracket,rrbrack]{stmaryrd} 

\theoremstyle{plain}
\newtheorem{theorem}{Theorem}[section]
\newtheorem{proposition}[theorem]{Proposition}
\newtheorem{lemma}[theorem]{Lemma}
\newtheorem{corollary}[theorem]{Corollary}

\theoremstyle{remark}
\newtheorem{remark}[theorem]{Remark}
\newtheorem{example}[theorem]{Example}

\theoremstyle{definition}
\newtheorem{definition}{Definition}[section]

\begin{document}
\title[Nelson algebras of rough sets induced by quasiorders]{Information completeness in Nelson algebras of rough sets induced by quasiorders}

\author[J.~J{\"a}rvinen]{Jouni J{\"a}rvinen}
\address{{\sc J.~J\"{a}rvinen}: Sirkkalankatu~6\\ 20520~Turku\\ Finland} 
\email{Jouni.Kalervo.Jarvinen@gmail.com}

\author[P.~Pagliani]{Piero Pagliani}
\address{{\sc P.~Pagliani}: Research Group on Knowledge and Communication Models\\ Via~imperia~6\\ 00161~Roma\\ Italy}
\email{p.pagliani@agora.it}

\author[S.~Radeleczki]{S{\'a}ndor Radeleczki}
\thanks{This research was carried out as part of the TAMOP-4.2.1.B-10/2/KONV-2010-0001 project supported by the European Union, co-financed by the European Social Fund, which
is gratefully acknowledged by S{\'a}ndor Radeleczki}
\address{{\sc S.~Radeleczki}: Institute of Mathematics\\ University of Miskolc\\ 3515~Miskolc-Egyetemv{\'a}ros\\ Hungary}
\email{matradi@uni-miskolc.hu}

\begin{abstract}
In this paper, we give an algebraic completeness theorem for constructive logic with 
strong negation in terms of finite rough set-based Nelson algebras determined by quasiorders. 
We show how for a quasiorder $R$, its rough set-based Nelson algebra can be obtained by applying 
the well-known construction by Sendlewski. We prove that if the set of all $R$-closed
elements, which may be viewed as the set of completely defined objects, is cofinal, then the rough set-based Nelson
algebra determined by a quasiorder forms an effective lattice, that is, an algebraic model of 
the logic $E_0$, which is characterised by a modal operator grasping the notion of ``to be classically valid''.
We present a necessary and sufficient condition under which a Nelson algebra is isomorphic to a rough 
set-based effective lattice determined by a quasiorder.
\end{abstract}


\maketitle

\section{Motivation: Mixing classical and non-classical logics}
\label{Sec:Intro}

Mixing logical behaviours is a more and more investigated topic in logic.
For instance, labelled deductive systems by D.~M.~Gabbay \cite{Gabbay97}
are used at this aim, and the ``stoup'' mechanism introduced by J-Y.~Girard
in \cite{Girard93} makes intuitionistic and classical deductions interact.

In 1989, P.~A.~Miglioli with his co-authors \cite{Miglioli89a} introduced a
constructive logic with strong negation, called \emph{effective logic zero}
and denoted by $E_{0}$, containing a modal operator
$\mathbf{T}$ such that for any formula $\alpha$ of $E_{0}$, $\mathbf{T} (\alpha)$
means that $\alpha$ is classically valid.
More precisely, given a Hilbert-style calculus for constructive logic
with strong negation (CLSN), also called Nelson logic \cite{Nelson49}, the rules for
{\bf T} are
\begin{center}
$({\sim} \alpha \rightarrow \perp) \rightarrow \mathbf{T}(\alpha)$ \qquad and \qquad
$(\alpha \rightarrow \perp) \rightarrow {\sim} \mathbf{T}(\alpha)$,
\end{center}
where $\sim$ denotes the \emph{strong negation}.
One obtains that $\alpha$ is valid in classical logic (CL) if and only if $\mathbf{T}(\alpha)$ is provable in $E_{0}$.
Therefore, $\mathbf{T}$ acts as an intuitionistic double negation $\neg\neg$ which, in view of the
G\"{o}del-Glivenko theorem, is able to grasp classical validity in the intuitionistic propositional
calculus (INT) by stating that $ \vdash_\mathrm{CL} \, \alpha$ if and only if
$\vdash_\mathrm{INT} {\neg\neg \alpha}$.

However, $\mathbf{T}$ fulfils additional distinct features. Firstly, CLSN is equipped with a weak
negation $\neg$, defined similarly to the intuitionistic negation. But, the combinations
$\neg\neg$, ${\sim}\neg$, or $\neg{\sim}$ are not able to cope with classical tautologies
(see \cite{pagliani2008geometry}, for example).
Secondly, consider the Kuroda formula
$\forall x \, \neg\neg \alpha(x) \to \neg\neg \forall x \, \alpha(x)$.
As noted in \cite{Miglioli89a}, it is an example of the divergence between double negation and
an operator intended to represent classical truth, because  the formula
$\forall x \, \mathbf{T} ( \alpha(x) )\to \mathbf{T} ( \forall x \, \alpha(x))$
should be intuitively valid if $\mathbf{T}$ represents classical truth. But the Kuroda formula
is unprovable in intuitionistic predicate calculus, while the above-presented $\mathbf{T}$-translation
(and some other translations, too) are provable even in the the predicative version of $E_0$.

The motivation of the logical system $E_{0}$ was to grasp two distinct aspects of computation in program
synthesis and specification: the algorithmic aspect and data. The latter are supposed to be given, not to be proved or
computed; in fact ``data'' is the Latin plural of ``datum'', which, literally, means ``given''.
A single undifferentiated logic is not a wise choice to cope with both aspects, therefore in $E_{0}$ 
there are two different
logics at work: a constructive logic,  representing algorithms, and classical logic, representing the behaviour of data.
Data are assumed not to be constructively analysable and this is connected to the problem of 
the meaning of an atomic formula from a constructive point of view.
Since the meaning of a formula is given by its construction, according to the constructivistic philosophy,
and since its construction depends on the logical structure
of the formula, the meaning of an atomic formula, which as such has no structure, is the atomic formula itself.

This is the solution adopted by Miglioli and others in \cite{Miglioli89b}.
In that paper, it is assumed that atomic formulas cannot have a constructive proof, therefore $p$ and
$\mathbf{T}(p)$ must coincide, that is, an axiom schema
\begin{equation} \label{Eq:Atomic}
 p \leftrightarrow \mathbf{T}(p) \tag{$\star$}
\end{equation}
is included for propositional variables.
Axiom \eqref{Eq:Atomic} together with the $\mathbf{T}$-version of the Kreisel-Putnam principle
\cite{KreiselPutnam}, that is,
\begin{equation}\tag{T-KP}\label{Eq:TKP}
(\mathbf{T}(\alpha) \rightarrow (\beta\vee\gamma) )
\rightarrow((\mathbf{T}(\alpha) \rightarrow \beta) \vee (\mathbf{T}(\alpha)
\rightarrow\gamma))
\end{equation}
characterises the logic ${\mathcal F}_{\rm CL}$.
Because of ($\star$) the logic ${\mathcal F}_{\rm CL}$ is not standard in the sense that
it does not enjoy uniform substitution.
However, its \emph{stable part}, that is, the part which is closed under uniform substitution, coincides with
a well-known maximal intermediate constructive logic, namely Medvedev's logic, a faithful interpretation of the
intuitionistic logical principles (see \cite{Medvedev62,Miglioli89b}).

One year later, P.~Pagliani \cite{Pagliani90} was able to exhibit an algebraic model
for $E_0$. It turned out that these models are a special kind of Nelson algebras,
called {\em effective lattices}.

The paper is structured as follows. In the next section we recall some well-known facts
about Heyting algebras, Nelson algebras, and effective lattices. In Section~\ref{Sec:RS_equivalence},
we recollecting some well-known results related to rough sets defined by equivalence relations and 
the semi-simple Nelson algebras they determine.
In Section~\ref{Sec:RS_quasiorders}, we recall the fact that rough set systems induced by quasiorders 
determine Nelson algebras, and show how these algebras can be obtained by Sendlewski's construction. 
We also present a completeness theorem for CLSN in terms of finite rough set-based Nelson algebras. 
We give several equivalent conditions under which rough set-based Nelson algebras form effective
lattices, and this enables us to characterize the Nelson algebras which are isomorphic to rough 
set-based effective lattices determined by quasiorders. Some concluding remarks of 
Section~\ref{Sec:Conclusions} end the work. In particular, it is shown how the logic $E_0$ can be interpreted
in terms of rough sets by following the very philosophy of rough set theory.

\section{Preliminaries: Heyting algebras, Nelson algebras, and \mbox{effective} lattices}
\label{Sec:Preliminary}

A \textit{Kleene algebra} is a structure $(A, \vee, \wedge, {\sim}, 0, 1)$ such that
$A$ is a 0,1-bounded distributive lattice and for all $a,b \in A$:
\begin{enumerate}[({K}1)]
\item ${\sim}\,{\sim}a  =  a $
\item $a \leq b  \text{ if and only if }  {\sim}b \leq {\sim}a$
\item $a \wedge {\sim}a  \leq  b \vee {\sim}b$
\end{enumerate}
A \textit{Nelson algebra}
$(A, \vee, \wedge, \rightarrow, {\sim}, 0, 1)$ is a Kleene algebra
$(A, \vee, \wedge, {\sim}, 0, 1)$ such that for all $a,b,c\in A$:
\begin{enumerate}[({N}1)]
\item $a\wedge c \leq {\sim} a\vee b$ if and only if $c\leq a\rightarrow b$,
\item $(a\wedge b)\rightarrow c = a \rightarrow (b\rightarrow c)$.
\end{enumerate}
In each Nelson algebra, an operation $\neg$ can be defined as
$\neg a = a \to 0$. The operation $\to$ is called \emph{weak relative
pseudocomplementation}, $\sim$ is called \emph{strong negation},
and $\neg$ is called \emph{weak negation}. A Nelson algebra
is \emph{semi-simple} if $a \vee \neg a = 1$ for all $a \in A$. It is well
known that semi-simple Nelson algebras coincide with three-valued {\L}ukasiewicz
algebras and regular double Stone algebras.

An element $a^*$ in a lattice $L$ with $0$ is called a \emph{pseudocomplement} 
of $a \in L$, if $a \wedge x =0 \iff x \leq a^*$ for all $x \in L$. 
If a pseudocomplement of $a$ exists, then it is unique, and a lattice in which every 
element has a pseudocomplement is called a \emph{pseudocomplemented lattice}. 
Note that pseudocomplemented lattices are always bounded. 
An element $a$ of pseudocomplemented lattice is 
\emph{dense} if $a^* = 0$. A \emph{Heyting algebra} $H$ is a lattice with $0$ 
such that for all $a,b \in H$, there is a greatest element $x$ of $H$ with $a \wedge x \leq b$.
This element is the \emph{relative pseudocomplement} of $a$ with respect to $b$,
and is denoted $a \Rightarrow b$. It is known that a complete lattice is a Heyting algebra 
if and only if it satisfies the \emph{join-infinite distributive law}, that is, finite meets distribute 
over arbitrary joins. In a Heyting algebra, the \emph{pseudocomplement} of $a$ 
is $a \Rightarrow 0$. By a \emph{double Heyting algebra} we mean a Heyting algebra
$H$ whose dual $H^d$ is also a Heyting algebra. A \emph{completely distributive lattice} is a complete
lattice in which arbitrary joins distribute over arbitrary meets. Therefore,
completely distributive lattices are double Heyting algebras.

A Heyting algebra $H$ can be viewed either as a partially ordered set $(H,\leq)$, because
the operations $\vee$, $\wedge$, $\Rightarrow$, $0$, $1$ are uniquely determined by the
order $\leq$, or as an algebra $(H,\vee,\wedge,\Rightarrow,0,1)$ of type $(2,2,2,0,0)$. 
Congruences on Heyting algebras are equivalences compatible with operations $\vee$, $\wedge$,
and $\Rightarrow$. Next we recall some well-known facts about congruences on Heyting algebras 
that can be found \cite{blyth2005lattices}, for instance.
Let $L$ be a distributive lattice and let $F$ be a filter of $L$. The equivalence
\[ \theta(F) = \{ (x,y) \mid (\exists z \in F) \, x \wedge z = y \wedge z \} \]
is a congruence on $L$. It is known that if $H$ is a Heyting algebra,
then $\theta(F)$ is a congruence on $H$, that is, $\theta(F)$ is compatible also
with $\Rightarrow$. Additionally, all congruences on Heyting algebras
are obtained by this construction.
A congruence on a Heyting algebra is said to be a \emph{Boolean congruence}
if its quotient algebra is a Boolean algebra. For a Heyting algebra $H$,
a filter $F$ contains the filter $D$ of all the dense elements of $H$ if and
only if $\theta(F)$ is a Boolean congruence. This means that $\theta(D)$
is the \emph{least} Boolean congruence on $H$, which is known as the \emph{Glivenko congruence} 
$\Gamma$, expressed also as 
\[
\Gamma = \{ (a,b) \mid a^* = b^* \}.
\]

\begin{lemma} \label{Lem:BooleanCongruencePseudo}
Let $H$ be a Heyting algebra and let $a \leq d$
for all dense elements $d$ of $H$. The equivalence
  \[ {\cong_a} = \{ (x,y) \mid  x \wedge a = y \wedge a \}
  \]
is a Boolean congruence on $H$.
\end{lemma}

\begin{proof}
Let $F_a = \{ x \in H \mid a \leq x\}$ be the principal filter of $a$.
Then $\theta(F_a)$ is a congruence on $H$, and clearly $\cong_a$ is equal to 
$\theta(F_a)$ (see e.g. \cite{pagliani2008geometry}). 
Because $a \leq d$ for all $d \in D$, we have  $D \subseteq F_a$ and so $\cong_a$ is a 
Boolean congruence.
\end{proof}

Let $\Theta$ be a Boolean congruence on a Heyting algebra $H$. 
As shown by A.~Sendlewski \cite{Sendlewski90}, the set of pairs
\begin{equation}\label{Eq:Nelson-construction}
N_\Theta (H)=
\{(a,b)\in H \times H \mid a\wedge b=0 \text{ and } a\vee b \, \Theta \, 1\}
\end{equation}
can be made into a Nelson algebra, if equipped with the operations:
\begin{align*}
(a,b) \vee(c,d)  & = (a \vee c, b \wedge d);\\ 
(a,b) \wedge(c,d)  & = (a \wedge c, b \vee d);\\
(a,b) \to(c,d)  & = (a \Rightarrow c, a \wedge d);\\ 
{\sim}(a,b)  & = (b,a).
\end{align*}
Note that $(0,1)$ is the $0$-element, $(1,0)$ is the $1$-element, and in the right-hand 
side of the above equations, the operations are those of the Heyting algebra
$H$. This Nelson algebra is denoted by $\mathbb{N}_{\Theta}(H)$.

In \cite{Pagliani90}, Pagliani introduced {\em effective lattices}. They
are special type of Nelson algebras determined by Glivenko congruences on
Heyting algebras, that is, for any Heyting algebra $H$ and its
Glivenko congruence $\Gamma$, the corresponding \emph{effective lattice}
is the Nelson algebra $\mathbb{N}_{\Gamma}(H)$. Note that for all
$x \in H$, $x \, {\Gamma} \, 1$ if and only if $x$ is dense. This
means that $N_\Gamma(H)$ consists of the pairs $(a,b)$ such
that $a \wedge b = 0$ and $a \vee b$ is dense (see Remark~2 in
\cite{Sendlewski90} and Proposition~9 of \cite{Pagliani90}).
 Additionally,  it is proved in \cite{Pagliani90} that
effective lattices are models for the logic $E_0$, with $\mathbf{T}$ defined by
$\mathbf{T}((a,b)) = (a^{**},b^{**})$. Note that each Heyting algebra
defines exactly one effective lattice, and that all effective lattices
are determined this way.

\section{Rough set theory comes into the picture}
\label{Sec:RS_equivalence}

Rough sets were introduced by Z.~Pawlak \cite{Pawl82} in order to provide a formal approach to deal
with incomplete data. In rough set theory, any set of entities (or points, or objects) comes with 
a lower approximation and an upper approximation. These approximations are defined on the basis of 
the attributes (or parameters, or properties) through which entities are observed or analysed.

Originally, in rough set theory it was assumed that the set of attributes induces an
equivalence relation $E$ on $U$ such that $x \,E \, y$ means that $x$ and $y$ cannot be
discerned on the basis of the information provided by their attribute values. Approximations
are then defined in terms of an {\em indiscernibility space}, that is,
a relational structure $(U,E)$ such that $E$ is an equivalence relation on $U$.
For a subset $X$ of $U$, the \emph{lower approximation} $X_{E}$ of $X$ consists of all
elements whose $E$-class is included in $X$, while the {\em upper approximation} $X^{E}$
is the set of the elements whose $E $-class has non-empty intersection with $X$. Therefore,
$X_{E}$ can be viewed as the set of elements which \emph{certainly} belong to $X$, and $X^{E}$
is the set of objects that \emph{possibly} are in $X$, when elements are  observed
through the knowledge synthesized by $E$.

Since inception, a number of generalisations of the notion of a rough set have been
proposed. A most interesting and useful one is the use of arbitrary binary relations
instead of equivalences. Let us now define approximations $(\cdot)_R$ and $(\cdot)^R$  
in a way that is applicable for different types of binary relations considered in this paper, 
and introduce  also other notions and notation we shall need. It is worth pointing out 
that $(\cdot)_R$ and $(\cdot)^R$ can be regarded as ``real'' lower and upper approximation operators, 
respectively, only if $R$ is reflexive, because otherwise $X_R \subseteq X$ 
and $X^R \subseteq X$ may fail to hold.
 
\begin{definition} \label{Def:BasicNotation}
Let $R$ be a reflexive relation on $U$ and $X \subseteq U$. The set 
$R(X) = \{y \in U \mid  x \, R\, y \text{ for some $x \in X$}\}$ is 
the \emph{$R$-neighbourhood} of $X$. If $X = \{a\}$, then we write $R(a)$ instead of $R(\{a\})$. 
The approximations are defined as $X_R = \{x \in U \mid R(x) \subseteq X\}$ 
and $X^R = \{x \in U \mid R(x) \cap X \neq \emptyset\}$. A set $X \subseteq U$
is called \emph{$R$-closed} if $R(X) = X$, and an element $x \in U$
is \emph{$R$-closed}, if its singleton set $\{x\}$ is $R$-closed. 
The set of $R$-closed points is denoted by $S$.
\end{definition}

Let us assume that $(U,E)$ is an indiscernibility space. 
The set of lower approximations $\mathcal{B}_{E}(U)=\{ X_{E} \mid X \subseteq U \}$ and
the set of upper approximations $\mathcal{B}^{E}(U)= \{ X^{E} \mid X \subseteq U \}$ 
coincide, so we denote this set simply by $\mathcal{B}_{E}(U)$. The set $\mathcal{B}_{E}(U)$ 
is a complete Boolean sublattice of $(\wp(U),\subseteq)$, where $\wp(U)$ denotes the
set of all subsets of $U$. This means that   $\mathcal{B}_{E}(U)$ forms a \emph{complete field of sets}. 
Complete fields of sets are in one-to-one correspondence with equivalence relations,
meaning that for each complete field of sets $\mathcal{F}$ on $U$, we can can define an equivalence 
$E$ such that $B_E(U) = \mathcal{F}$. Note that $S$ and all its subsets belong to 
$\mathcal{B}_{E}(U)$, meaning that $\wp(S)$ is a complete sublattice of $\mathcal{B}_{E}(U)$,
and therefore in this sense $S$ can be viewed to consist of \emph{completely defined objects}.
Each object in $S$ can be separated from other points of $U$ by the information 
provided by the indiscernibility relation $E$, meaning that
for any $x \in S$ and $X \subseteq U$, $x \in X_E$ if and only if $x \in X^E$.

The {\em rough set} of $X$ is the equivalence class of all $Y\subseteq U$ such that $Y_E=X_E$ and $Y^E=X^E$.
Since each rough set is uniquely determined by the approximation pair, one can represent the rough set of 
$X$ as $(X_E,  X^E)$ or $(X_E,-X^E)$. We call the former \emph{increasing representation} and the latter 
\emph{disjoint representation}. These representations induce the sets
\[
\mathit{IRS}_E(U)  = \{(X_E,X^E)\mid X\subseteq U\}
\]
and
\[
\mathit{DRS}_E(U)  =\{(X_E,-X^E)\mid X\subseteq U\},
\]
respectively. The set $\mathit{IRS}_E(U)$ can be ordered pointwise
\[(X_E,X^E) \leq (Y_E,Y^E) \iff X_E \subseteq Y_E \text{ and } X^E \subseteq Y^E,\]
and $\mathit{DRS}_E(U)$ is ordered by reversing the order for the second components
of the pairs, that is,
\[(X_E,-X^E) \leq (Y_E,-Y^E) \iff X_E \subseteq Y_E \text{ and } -X^E \supseteq -Y^E.\]
Therefore,  $\mathit{IRS}_E(U)$ and  $\mathit{DRS}_E(U)$ are order-isomorphic, and they
form completely distributive lattices, thus double Heyting algebras
\cite{Pagliani93,Pagliani97,pagliani2008geometry}.

Every Boolean lattice $B$, where $x'$ denotes the complement of $x \in B$,   
is a Heyting algebra such that $x \Rightarrow y = x' \vee y$ for $x,y \in B$. 
An element $x \in B$ is dense only if $x' = 0$, that is, $x = 1$. Because it is known that on a 
Boolean lattice each lattice-congruence is such that the quotient lattice
is a Boolean lattice, also the congruence $\cong_S$ on  $\mathcal{B}_{E}(U)$, 
defined by $X \cong_S Y$, if $X \cap S = Y \cap S$, is Boolean when
$\mathcal{B}_{E}(U)$ is interpreted as a Heyting algebra.

In \cite{Pagliani93}, it is shown that disjoint representation of rough sets can be characterized as
\begin{equation}\label{Eq:DRS-construction}
\mathit{DRS}_E(U) = \{(A,B) \in \mathcal{B}_E(U)^2 \mid A\cap B =\emptyset\ \text{ and }  A \cup B \cong_{S}U\}.
\end{equation}
Thus, ${DRS}_E(U)$ coincides with the Nelson lattice $N_{\cong_S} (\mathcal{B}_E(U))$.
Since $\mathcal{B}_{E}(U)$ is a Boolean lattice, ${\mathbb N}_{\cong_S} (\mathcal{B}_E(U))$
is a semi-simple Nelson algebra (cf. \cite{Pagliani96}).
As a consequence, we obtain the well-known facts that
rough sets defined by equivalences determine also regular double Stone algebras and three-valued
{\L}ukasiewicz algebras.

In the literature also several representation theorems related to
rough sets induced by equivalences can be found. For instance, in \cite{Pagliani97} it was proved that for any
{\em finite} three-valued \L ukasiewicz algebra $\mathbb A$, there is an indiscernibility space $(U,E)$ such that
$\mathbb{N}_{\cong_S}(\mathcal{B}_{E}(U))$ is isomorphic to $\mathbb{A}$.
This result was extended by L.~Iturrioz \cite{Iturrioz99} by showing that any three-valued
{\L}ukasiewicz algebra is a subalgebra of $\mathit{IRS}_E(U)$ for some indiscernibility space $(U,E)$.
Finally, it has been proved by J.~J\"{a}rvinen and S.~Radeleczki \cite{JarRad} that for any semi-simple
Nelson algebra $\mathbb{A}$ with an underlying algebraic lattice there exists an indiscernibility space
$(U, E)$ such that $\mathbb{A}$ is isomorphic to $\mathbb{N}_{\cong_S}(\mathcal{B}_{E}(U))$.
From the latter representation theorem one obtains that every semi-simple Nelson algebra,
regular double Stone algebra and three-valued {\L}ukasiewicz algebra that are
defined on algebraic lattices can be obtained from an indiscernibility space $(U,E)$ by
using Sendlewski's construction \eqref{Eq:Nelson-construction}.
Note that an \emph{algebraic lattice} is a complete lattice $L$ such that each element $x$ of $L$ is the join of a set of
compact elements of $L$, and thus finite lattices are trivially algebraic. 

On  $\mathcal{B}_{E}(U)$, the Glivenko congruence is simply
the identity relation. This means that the effective lattice determined by the indiscernibility
space $(U,E)$ is just the collection of all ordered pairs of disjoint elements of $\mathcal{B}_{E}(U)$
such that $X \cup Y = U$. Hence, $\mathbb{N}_{\Gamma}(\mathcal{B}_{E}(U))$ equals the set
of pairs $\{(X,-X) \mid X \in \mathcal{B}_{E}(U)\}$, which trivially is an isomorphic copy of
$\mathcal{B}_{E}(U)$ itself. Therefore, in the case of equivalence relations, effective
lattices do not appear that interesting, because on  $\mathbb{N}_{\Gamma}(\mathcal{B}_{E}(U))$
for any formula $\alpha$ we would have ${\bf T}(\llbracket\alpha\rrbracket)=\llbracket\alpha\rrbracket$,
where $\llbracket\alpha\rrbracket$ is the ordered pair interpreting $\alpha$.

Then a question arises: {\em Is there
any generalization which makes it possible to go ahead and develop a full correspondence
between rough set systems and effective lattices?}

\section{Effective lattices and quasiorders}
\label{Sec:RS_quasiorders}

For a quasiorder $R$ on $U$, as in case of equivalences, we may define
the \emph{increasing representation} and the \emph{disjoint representation},
respectively, by
\[
\mathit{IRS}_R(U)  = \{(X_R,X^R)\mid X\subseteq U\}
\]
and
\[
\mathit{DRS}_R(U)  =\{(X_R,-X^R)\mid X\subseteq U\},
\]
and these sets can be identified by the bijection $(X_R,X^R) \mapsto (X_R,-X^R)$.

As shown by J.~J\"{a}rvinen, S.~Radeleczki, and L.~Veres \cite{JRV09}, $\mathit{IRS}_{R}(U)$ is a complete 
sublattice of $\wp(U) \times\wp(U)$ ordered by the pointwise set-inclusion relation, meaning that $\mathit{IRS}_R(U)$ 
is an algebraic completely distributive lattice such that
\[
\bigwedge\left\{  ( X_{R}, X^{R} ) \mid X\in\mathcal{H} \right\}  =
\Big ( \bigcap_{X \in\mathcal{H}} X_{R}, \bigcap_{X\in\mathcal{H}} X^{R}
\Big )
\]
and
\[
\bigvee\left\{  ( X_{R}, X^{R} ) \mid X\in\mathcal{H} \right\}  =
\Big ( \bigcup_{X\in\mathcal{H}} X_{R}, \bigcup_{X\in\mathcal{H}} X^{R} \Big )
\]
for all $\mathcal{H} \subseteq \mathit{IRS}_{R}(U)$. 
Since $\mathit{IRS}_{R}(U)$ is completely distributive, it is a double Heyting algebra. 

J\"{a}rvinen and Radeleczki proved in \cite{JarRad} that the bounded distributive
lattice $\mathit{IRS}_R(U)$ equipped with the operation $\sim$ defined by ${\sim}(X_R,X^R) = (-X^R,-X_R)$ 
forms a Kleene algebra satisfying the interpolation property of \cite{Cign86}.
It is proved by R.~Cignoli \cite{Cign86} that any Kleene algebra that satisfies this
interpolation property and is such that for each pair $a$ and $b$ of its elements, 
the relative pseudocomplement $a \Rightarrow {\sim} a \vee b$ exists, forms a Nelson algebra
in which the operation $\to$ is determined by the rule $a \to b := a \Rightarrow {\sim} a \vee b$.
Therefore, as noted in \cite{JarRad}, for any quasiorder $R$ on $U$, $\mathit{IRS}_R(U)$ 
together with the operation $\sim$ forms a Nelson algebra. This Nelson algebra 
is denoted by  $\mathbb{IRS}_R(U)$, and its operations
will be described explicitly in Corollary~\ref{Cor:IRS_Sendlewski}.

In \cite{JarRad}, it is also proved that if $\mathbb{A}$ is a Nelson algebra defined on an algebraic lattice,
then there exists a set $U$ and a quasiorder $R$ on $U$ such that $\mathbb{A}$ and the Nelson algebra
$\mathbb{IRS}_{R}(U)$ are isomorphic. From this, we can deduce the following
completeness result, with the finite model property, for CLSN, whose axiomatisation can be found in
\cite{Rasiowa74,Vaka77}, for example.

\begin{theorem} \label{Thm:CompleteCLSN}
Let $\alpha$ be a formula of CLSN. Then the following conditions are equivalent:
\begin{enumerate}[\rm (a)]
 \item $\alpha$ is a theorem,
 \item $\alpha$ is valid in every finite rough set-based Nelson algebra determined by a quasiorder.
\end{enumerate}
\end{theorem}

\begin{proof} Suppose that $\alpha$ is a theorem. Then, in the view of the completeness theorem
 proved in H.~Rasiowa \cite{Rasiowa74}, $\alpha$ is valid in every Nelson algebra. Particularly,
 $\alpha$ is valid in every finite rough set-based Nelson algebra determined by a quasiorder.

 Conversely, assume $\alpha$ is not a theorem. Then, there exists a finite
 Nelson algebra $\mathbb{A}$ such that $\alpha$ is not valid in that algebra,
 that is, its valuation $\llbracket \alpha \rrbracket_\mathbb{A}$ is different from $1_\mathbb{A}$ 
 (see \cite{Vaka77}, for example).  Because $\mathbb{A}$ is finite, it is defined on an algebraic lattice. Therefore, there
 exists a finite set $U$ and a quasiorder $R$ on $U$ such that $\mathbb{A}$ and the finite rough set-based Nelson algebra
 $\mathbb{IRS}_R(U)$ determined by $R$ are isomorphic. We denote here $\mathbb{IRS}_R(U)$ simply by $\mathbb{IRS}$.
Let us denote by $f$ the isomorphism between these
 Nelson algebras. The valuation on $\mathbb{IRS}$ can be now defined as
 $\llbracket \beta \rrbracket_{\mathbb{IRS}} =  f(\llbracket  \beta \rrbracket_\mathbb{A})$
 for all formulas $\beta$, so $\llbracket \alpha \rrbracket_{\mathbb{IRS}}
 = f(\llbracket \alpha \rrbracket_\mathbb{A}) \neq f(1_\mathbb{A}) = 1_{\mathbb{IRS}}$, that is,
 $\alpha$ is not valid in $\mathbb{IRS}$.
\end{proof}

An element $x$ of a complete lattice $L$ is \emph{completely join-irreducible}
if for every subset $X$ of $L$, $x = \bigvee X$ implies that $x \in X$.
The set of completely join-irreducible elements of $L$ is
denoted by $\mathcal{J}$. It is shown in \cite{JRV09}
that the set of completely join-irreducible elements of $\mathit{IRS}_{R}(U)$ is
\[
\mathcal{J} = \{(\emptyset,\{x\}^{R}) \mid  x \in U \text{ and } |R(x)| \geq2 \} \cup\{( R(x),
R(x)^{R}) \mid x \in U\},
\]
and that every element can be represented as a join of elements in
$\mathcal{J}$. 

In a Nelson algebra $\mathbb{A}$ defined on an algebraic lattice $(A,\leq)$,
each element is the join of completely join-irreducible elements $\mathcal{J}$.
We may define for any $j \in \mathcal{J}$ the element 
$g(j) = \bigwedge \{ x \in A \mid x \nleq {\sim} j \} \ (\in \mathcal{J})$,
and it is shown in \cite{JarRad} that the map $g \colon \mathcal{J} \to \mathcal{J}$
satisfies the following conditions for all $x,y \in \mathcal{J}$:
\begin{enumerate}[({J}1)]
 \item if $x \leq y$, then $g(y) \leq g(x)$,
 \item $g(g(x)) = x$,
 \item $x \leq g(x)$ or $g(x) \leq x$,
 \item if $x,y \leq g(x),g(y)$, there exists $z\in \mathcal{J}$ such that $x,y \leq z \leq g(x),g(y)$.
\end{enumerate}
Conversely, the mapping $g$ determines the strong negation $\sim$ on $\mathbb{A}$ by the equation
${\sim} x = \bigvee \{ j \in \mathcal{J} \mid g(j) \nleq x\}$.

Let $R$ be a quasiorder on $U$ and let $\mathcal{J}$ be the set of the completely join-irreducible
elements of $\mathit{IRS}_R(U)$. In \cite{JarRad} it is proved 
that the following equations hold:
\begin{align*}
\{ j \in \mathcal{J} \mid j < g(j) \} & = \{(\emptyset,\{x\}^R) \mid x \in U \text{ and } |R(x)| \geq 2 \}, \\
\{ j \in \mathcal{J} \mid j = g(j) \} & =  \{( \{x\}, \{x\}^R) \mid \ x \in S \, \}, \\
\{ j \in \mathcal{J} \mid j > g(j) \} & = \{( R(x), R(x)^R) \mid  x \in U \text{ and } |R(x)| \geq 2 \}.
\end{align*}
Therefore, elements in $S$ have a special role, because they are such that the completely
join-irreducible elements corresponding to them are the fixed points of $g$.
It should be also noted that in case of an equivalence $E$, the partially ordered set of completely 
join-irreducible elements of  $\mathit{IRS}_E(U)$ consists of disjoint chains of 1 and 2 elements.

Differently from an equivalence $E$ that defines one  complete field of sets $\mathcal{B}_E(U)$, 
a quasiorder $R$ determines two \emph{complete rings of sets}, or equivalently, two
\emph{Alexandrov topologies}
\[
\mathcal{T}_{R}(U) = \{ X_{R} \mid X \subseteq U\} \text{ \ and \ }
\mathcal{T}^{R}(U) = \{ X^{R} \mid X \subseteq U\}, 
\]
that is, $\mathcal{T}_{R}(U)$ and $\mathcal{T}^{R}(U)$ are closed under arbitrary
unions and intersections. Note that $\mathcal{T}_{R}(U)$ and $\mathcal{T}^{R}(U)$ 
are intended to be the open sets of these topologies, respectively. The Alexandrov
topologies $\mathcal{T}_{R}(U)$ and $\mathcal{T}^{R}(U)$ are dual in the sense that 
$X \in\mathcal{T}_{R}(U)$ if and only if $-X \in\mathcal{T}^{R}(U)$. 

The topology $\mathcal{T}_{R}(U)$ consists of all $R$-closed sets. 
Therefore, for any $X \subseteq U$, the $R$-neighbourhood
$R(X)$ of $X$ is the smallest open set containing $X$. This actually means that
$\mathcal{T}_{R}(U) = \{ R(X) \mid X \subseteq U\}$, which also implies
$R(X)_R = R(X)$ and $R(X_R) = X_R$ for any $X \subseteq U$. 
In addition, for all $X \in \mathcal{T}_R(U)$, $X = \bigcup_{x \in X} R(x)$
(see \cite{Jarv07}, for instance).

Because the points of $S$ are $R$-closed, $\wp(S)$ is a complete sublattice of 
$\mathcal{T}_{R}(U)$, as in case of equivalences. Again, each object in $S$ can 
be separated from other points of $U$ by the information provided by the 
relation $R$, because each element of $S$ is $R$-related only to itself, and
for any $x \in S$ and $X \subseteq U$, $x \in X_R$ if and only if $x \in X^R$. 
Therefore, also in case of quasiorders, $S$ can be viewed as the set of
\emph{completely defined objects}.

In the Alexandrov topology $\mathcal{T}_{R}(U)$, the map $(\cdot)^{R} \colon\wp(U) \to\wp(U)$ is the 
closure operator and $(\cdot)_{R} \colon\wp(U) \to\wp(U)$ is the interior operator. Because
$\mathcal{T}_{R}(U)$ is closed under arbitrary unions and intersections, it is a
completely distributive lattice and a double Heyting algebra. In particular,
for any $X,Y \in \mathcal{T}_{R}(U)$, the relative pseudocomplement 
$X \Rightarrow Y$ equals $(-X \cup Y)_{R}$. Thus, the structure
\begin{equation}\label{Eq:Heyting}
(\mathcal{T}_{R}(U),\cup,\cap,\Rightarrow,\emptyset,U)
\end{equation}
forms a Heyting algebra in which $X^* = (-X)_R = -X^R$. Hence, an
element $X \in \mathcal{T}_{R}(U)$ is dense if and only if $X^R = U$, meaning that $X$ is 
\emph{cofinal} in $U$, that is, for any $x \in U$, there exists $y \in X$ such that 
$x \, R \, y$ (see \cite{Stone68} for more details on cofinal sets).

Because each increasing rough set pair belongs to $\mathcal{T}_{R}(U) \times\mathcal{T}^{R}(U)$,
our next aim is to present a characterization of $\mathit{IRS}_{R}(U)$
in terms of pairs belonging to $\mathcal{T}_{R}(U) \times\mathcal{T}^{R}(U)$.
The next proposition appeared for the first time in \cite{JarPagRad11}, and also an analogous result is
presented independently in \cite{NagaUma11}.

\begin{proposition}\label{Prop:characterization}
Let $R$ be a quasiorder on $U$. Then,
\[
\mathit{IRS}_R(U) = \{(A,B) \in \mathcal{T}_{R}(U) \times\mathcal{T}^{R}(U)
\mid A \subseteq B \text{ and } S \subseteq A \cup-B \}.
\]
\end{proposition}

\begin{proof} ($\subseteq$): Suppose $(X_R,X^R) \in \mathit{IRS}_R(U)$. Then, $X_R \subseteq X^R$.
Suppose now $x \in S$ and $x \notin X_R \cup -X^R$. Then $x \in X^R \setminus X_R$,
which is clearly impossible because $x \in S$. Thus,
$S \subseteq  X_R \cup -X^R$.

\medskip\noindent%
($\supseteq$): Assume that $(A,B)  \in \mathcal{T}_R(U) \times \mathcal{T}^R(U)$,
$A \subseteq B$, and  $S \subseteq A \cup -B$,
This means that $B \setminus A \subseteq -S$, that is, for any $x \in B \setminus A$, we have $|R(x)| \geq 2$.
For any $x \in B \setminus A$,
the pair $(\emptyset, \{x\}^R)$ is a rough set, because  $|R(x)| \geq 2$. Additionally, for any
$x \in A$, the pair $( R(x), R(x)^R)$ is also a rough set.
Let us consider the rough set
\begin{align*}
(C,D) &= \bigvee \{ (\emptyset,\{x\}^R) \mid x \in B \setminus A \} \vee \bigvee  \{( R(x), R(x)^R) \mid x \in A\} \\
& = \Big ( \bigcup_{x \in A} R(x), \bigcup \{ \{x\}^R \mid x \in B \setminus A \}
\cup \bigcup \{ R(x)^R \mid x \in A \} \Big).
\end{align*}
Clearly, $C = \bigcup_{x \in A} R(x) = A$ since $A \in  \mathcal{T}_R(U)$. In turn,
\[
D =  \bigcup \{ \{x\}^R \mid x \in B \setminus A \} \cup \bigcup \{ R(x)^R \mid x \in A \}.
\]
Now, in view of the fact that $A$ is $R$-closed,
and that $B$ is an upper approximation, hence $B^R=B$, we have:
\begin{enumerate}[(i)]
\item If $x \in A$, then $R(x) \subseteq A$, so $R(x)^R \subseteq A^R \subseteq B^R = B$.
\item If $x \in B \setminus A$, then  $\{x\}^R \subseteq B^R = B$.
\end{enumerate}
Therefore, $D \subseteq B$. Conversely, let $y \in B$. Then, $y \in A$ or $y \in B \setminus A$.
\begin{enumerate}[(i)]
\item If $y \in A$, then $y \in R(y)^R \subseteq D$.
\item If $y \in B \setminus A$, then $y \in \{y\}^R$ and
$y \in \bigcup \{ \{x\}^R \mid x \in B \setminus A \} \subseteq D$.
\end{enumerate}
Thus, we have shown $B = D$. Therefore, $(A,B) = (C,D)$ is a rough set, that is, $(A,B) \in \mathit{IRS}_R(U)$.
\end{proof}

As in the case of equivalences, it is obvious by Proposition~\ref{Prop:characterization} that
\begin{equation}\label{Eq:R-disjoint}
\mathit{DRS}_R(U) = \{ (A,B) \in \mathcal{T}_{R}(U) \times\mathcal{T}_{R}(U)
\mid A \cap B = \emptyset\text{ and } S \subseteq A \cup B \}.
\end{equation}
We can now connect rough sets defined by quasiorders to Sendlewski's construction
\eqref{Eq:Nelson-construction}. First, we need the following lemma.
\begin{lemma} \label{Lem:S=Dense}
The set $S$ is included in all dense elements of $\mathcal{T}_{R}(U)$.
\end{lemma}

\begin{proof} Suppose that the set $X \in \mathcal{T}_{R}(U)$ is dense, that is,
$X^R = U$. If $S \nsubseteq X$, then there exists $x \in S$ such that $x \notin X$.
Because $x \in X^R$ and $R(x) = \{x\}$, we have $x \in X$, a contradiction.
\end{proof}

Because $\mathcal{T}_{R}(U)$ forms a Heyting algebra \eqref{Eq:Heyting},
by the previous lemma and Lemma~\ref{Lem:BooleanCongruencePseudo},
$\cong_S$ is a Boolean congruence on the Heyting algebra $ \mathcal{T}_{R}(U)$. 
It is easy to see that for all $X \in \mathcal{T}_{R}(U)$, 
$X \cong_S U$ if and only if $S \subseteq X$. Therefore, by \eqref{Eq:R-disjoint}, we may write
\[
\mathit{DRS}_R(U) = N_{\cong_S}(\mathcal{T}_{R}(U)).
\]
By applying Sendlewski's construction \eqref{Eq:Nelson-construction}, we may now write the 
following proposition.

\begin{proposition}\label{Prop:DRS_Sendlewski}
If $R$ is a quasiorder on $U$, then $\mathit{DRS}_R(U)$ forms a Nelson algebra
with the operations:
\begin{align*}
 (X_R,-X^R) \vee  (Y_R,-Y^R)   & = (X_R \cup Y_R, -X^R \cap -Y^R): \\
 (X_R,-X^R) \wedge (Y_R,-Y^R) & = (X_R \cap Y_R, -X^R \cup -Y^R); \\
 {\sim}(X_R,-X^R) &= (-X^R, X_R);\\
  (X_R,-X^R)  \to  (Y_R,-Y^R)  &= ( (-X_R \cup Y_R)_R, X_R \cap -Y^R).
\end{align*}
\end{proposition}
We denote this Nelson algebra on $\mathit{DRS}_R(U)$ by $\mathbb{DRS}_R(U)$.
Because the map $(X_R,X^R) \mapsto (X_R,-X^R)$ is an order-isomorphism between complete lattices
$\mathit{IRS}_R(U)$ and $\mathit{DRS}_R(U)$, we may write the following corollary
describing the operations in the Nelson algebra $\mathbb{IRS}_R(U)$

\begin{corollary}\label{Cor:IRS_Sendlewski}
For a quasiorder $R$ on $U$, the operations of $\mathbb{IRS}_R(U)$ are: 
\begin{align*}
 (X_R,X^R) \vee  (Y_R,Y^R)   & = (X_R \cup Y_R, X^R \cup Y^R); \\
 (X_R,X^R) \wedge (Y_R,Y^R) & = (X_R \cap Y_R, X^R \cap Y^R); \\
 {\sim}(X_R,X^R) &= (-X^R, -X_R);\\
  (X_R,X^R)  \to  (Y_R,Y^R)  &= ( (-X_R \cup Y_R)_R, -X_R \cup Y^R). 
\end{align*}
\end{corollary}

We are now ready to consider effective lattices determined by rough sets.
Recall that for any Heyting algebra $H$, the corresponding effective lattice is
$\mathbb{N}_\Gamma (H)$, where $\Gamma$ is the Glivenko congruence on $H$. In Section~\ref{Sec:Preliminary}
we showed that every element $a \in H$ which is below all dense elements $D$ determines a Boolean
congruence $\cong_a$. If such an $a$ is itself a dense element, it must
be the least dense element, that is, $a = \bigwedge D \in D$. 
Therefore, in this case $\Gamma$ is equal both to $\cong_a$ and to the congruence $\theta(F_a)$ 
of the principal filter $F_a = \{x \in H \mid a \leq x\} = D$. \label{Page:Dense}

It should be noted that Heyting algebras do not necessarily have a least dense element. For instance, the
Heyting algebra defined on the real interval $[0,1] = \{ x \in \mathbb{R} \mid 0 \leq x \leq 1\}$
is such, because each non-zero element of the algebra is dense. On the contrary, in case of finite Heyting
algebras, there exists always the least dense element $\bigwedge D$, and thus $D$ is
the principal filter of $\bigwedge D$.

By definition, $\mathbb{DRS}_R(U)$ is an effective lattice whenever $\cong_S$ 
is the least Boolean congruence on the Heyting algebra $\mathcal{T}_R(U)$. Because the
Nelson algebras $\mathbb{DRS}_R(U)$ and $\mathbb{IRS}_R(U)$ are essentially the same, 
we say that also $\mathbb{IRS}_R(U)$ is an effective lattice, if $\cong_S$ is the 
Glivenko congruence of  $\mathcal{T}_R(U)$.

Our next lemma characterizes the conditions under which rough set-based Nelson algebras determined 
by quasiordes are effective lattices. 

\begin{proposition} \label{Prop:Dense}
Let $R$ be a quasiorder on the set $U$ and let $S$ be the set of $R$-closed elements.
The following statements are equivalent:
\begin{enumerate}[\rm (a)]
\item $S$ is cofinal in $U$,
\item $S$ is a dense element of the Heyting algebra $\mathcal{T}_{R}(U)$,
\item $S$ is the least dense element of the Heyting algebra $\mathcal{T}_{R}(U)$,
\item $\cong_S$ is the least Boolean congruence $\Gamma$ on the Heyting algebra $\mathcal{T}_{R}(U)$,
\item $\mathbb{IRS}_R(U)$ and $\mathbb{DRS}_R(U)$ are effective lattices.
\end{enumerate}
\end{proposition}

\begin{proof} Claims (a) and (b) are equivalent by definition, and the same holds
between (d) and (e). Trivially (c) implies (b), and by Lemma~\ref{Lem:S=Dense},
$S$ is included in each dense element of $\mathcal{T}_R(U)$, hence (b) implies (c).

If\, $\cong_S$ equals $\Gamma$, then $S \cong_S U$ implies 
$S \, \Gamma \, U$ and $X^* = U^* = \emptyset$, that is, $X$ is dense and 
(d)$\Rightarrow$(b).  If $S$ is the least dense set, then, as discussed earlier, 
${\cong_S}$ equals $\Gamma$ and (c)$\Rightarrow$(d).
\end{proof}

If ${\cong_S}$ is the Glivenko congruence, then for all elements $A,B \in \mathcal{T}_R(U)$, 
$A \cup B \cong_S U \iff A \cup B \text{ is dense } \iff (A \cup B)^R = 
A^R \cup B^R = U$. Therefore, we can write the following corollary
characterizing the elements of $\mathbb{DRS}_R(U)$ and $\mathbb{IRS}_R(U)$ 
in the case they are effective lattices.

\begin{corollary} Let $R$ be a quasiorder on $U$ and assume that $S$ is dense.
Then, the following equations hold:
\begin{enumerate}[\rm (a)]
 \item $\mathit{DRS}_R(U) = \{ (A,B) \in \mathcal{T}_{R}(U) \times\mathcal{T}_{R}(U)
\mid A \cap B = \emptyset\text{ and } A^R \cup B^R = U \}$;
\item $\mathit{IRS}_R(U) = \{(A,B) \in \mathcal{T}_{R}(U) \times\mathcal{T}^{R}(U)
\mid A \subseteq B \text{ and } B_R \setminus A^R = \emptyset \}$.
\end{enumerate}
\end{corollary}

Next, we consider shortly the case that $R$ is a partial order.
The well-known \emph{Hausdorff maximal principle} states that
in any partially ordered set, there exists a maximal chain. 

\begin{corollary} \label{Cor:BoundenChain}
If $(U,\leq)$ is a partially  ordered set such that any maximal chain is
bounded from above, then $\mathbb{IRS}_\leq(U)$ and $\mathbb{DRS}_\leq(U)$
are effective lattices.
\end{corollary}

\begin{proof}
Let $(U,\leq)$ be a partially ordered set and $x \in U$. Let us consider the partially
ordered set $(\{y \mid x \leq y\}, \leq_x)$, where $\leq_x$ is the order $\leq$ restricted to
$\{y \mid x \leq y\}$. Then, by the Hausdorff maximal principle, $\{y  \mid x \leq y\}$
has a maximal chain $C$.  By our assumption, the chain $C$ is bounded from above by
some element $m$. Because $m$ is a maximal element, it is in $S$ and $x\leq m$.
This implies $S^{\leq} = U$, that is, $S$ is dense. Therefore, $\cong_S$ is the least 
Boolean congruence $\Gamma$, and $\mathbb{IRS}_{\leq}(U)$ and $\mathbb{DRS}_{\leq}(U)$ 
are effective lattices.
\end{proof}

\begin{remark}
Clearly, if $U$ is finite, then Corollary~\ref{Cor:BoundenChain} holds, that is,
all rough set-based Nelson algebras determined by finite partially ordered sets are effective lattices.
\end{remark}

\begin{example}
Let $(U,\leq)$ be a partially ordered set with least element $0$ such that 
$U \setminus \{0\}$ is an antichain, that is, all elements in $U \setminus \{0\}$ are incomparable.
Clearly, the set of $\leq$-closed elements is $S = U \setminus \{0\}$ and $S^\leq = U$,
meaning that $S$ is cofinal and, by Lemma~\ref{Prop:Dense}, $\cong_S$ equals $\Gamma$ and 
$S$ is the least dense set. Additionally, $\mathcal{T}_\leq(U) = \wp(S) \cup \{U\}$, and
the congruence classes of $\cong_S$ are of the form $\{ X, X \cup \{0\} \}$, where $X \in \mathcal{T}_\leq(U)$.

It is also easy to observe that
\[ \textit{IRS}_\leq(U) = \{ (X,X\cup \{0\}) \mid X \subseteq S \, \}
 \cup \{ (\emptyset,\emptyset), (U,U) \}.
\]
So, in this case $\textit{IRS}_\leq(U)$ is order-isomorphic to $\wp(S)$ added with
a top element $\mathbf{1}$ corresponding to $(U,U)$ and a bottom element $\mathbf{0}$
corresponding to $(\emptyset,\emptyset)$, that is, $\textit{IRS}_\leq(U)$ is order-isomorphic
to $\mathbf{0} \oplus \wp(S) \oplus \mathbf{1}$. Note also that
in $\textit{IRS}_\leq(U)$, the pair $(\emptyset,\{0\})$ is the least dense set, which
means that in  $\textit{IRS}_\leq(U)$, all elements except $(\emptyset,\emptyset)$ are
dense.
\end{example}

We end this section by a presenting a necessary and sufficient condition under which 
Nelson algebras are isomorphic to effective lattices of rough sets determined by quasiorders.

\begin{theorem}\label{Thm:Representation}
Let $\mathbb{A}$ be a Nelson algebra. Then, there exists a set $U$ and a quasiorder 
$R$ on $U$ such that $\mathbb{IRS}_R(U)$ is an effective lattice and 
$\mathbb{A} \cong \mathbb{IRS}_R(U)$ if and only if $\mathbb{A}$
is defined on an algebraic lattice in which each completely 
join-irreducible element is comparable with at least one completely 
join-irreducible element which is a fixed point of $g$.
\end{theorem}

\begin{proof}
For a Nelson algebra $\mathbb{A}$, there exists a set $U$ and a quasiorder $R$ on $U$ such that 
$\mathbb{A} \cong \mathbb{IRS}_R(U)$ if and only if $\mathbb{A}$ is defined on an 
algebraic lattice (for details, see \cite{JarRad}). Additionally, by Lemma~\ref{Prop:Dense}, 
we know that $\mathbb{IRS}_R(U)$ is an effective lattice if and only 
if $S$ is cofinal in $U$.

Assume that there exists a set $U$ and a quasiorder $R$ on $U$ such that 
$\mathbb{IRS}_R(U)$ is an effective lattice and $\mathbb{A} \cong \mathbb{IRS}_R(U)$.
Let $\varphi$ be the isomorphism in question. This implies that  
$\mathbb{A}$ is defined on an algebraic lattice $A$ and each
element of $A$ can be represented as a join of completely irreducible
elements of $\mathcal{J}$. Note that $\varphi$ preserves also the map $g$, that is, $\varphi(g(j))
= g(\varphi(j))$ for all $j \in \mathcal{J}$; see \cite{JarRad}.

Let $j \in \mathcal{J}$. If $j$ is a fixed point of $g$, then
$\varphi(j) = ( \{y\}, \{y\}^R)$ for some $y \in S$, and
we have nothing to prove since $j$ is trivially comparable with itself.
If $j$ is not a fixed point of $g$, then for $\varphi(j)$ we have two possibilities:
\begin{enumerate}[(i)]
 \item $\varphi(j) = (\emptyset,\{x\}^R)$ for some $x \in U$ such that $|R(x)| \geq 2$, or 
 \item $\varphi(j) = ( R(x), R(x)^R) )$ for some $x \in U$ such that  $|R(x)| \geq 2$.
\end{enumerate}

Without a loss of generality we may assume that $j < g(j)$. This means that 
there exists $x \in U \setminus S$ such that $\varphi(j) =(\emptyset,\{x\}^R)$ 
and  $\varphi(g(j)) = ( R(x), R(x)^R) )$. Because $S$ is cofinal, 
there exists $y \in S$ such that $x \, R \, y$.
Let $k$ be the element of $A$ such that $\varphi(k) = (\{y\},\{y\}^R)$. 
Obviously, $k \in \mathcal{J}$ and $g(k) = k$. Because $x \, R \, y$, we have
$\varphi(j) = (\emptyset,\{x\}^R) \leq (\{y\},\{y\}^R) = \varphi(k)$, and
hence $j \leq k$; note that this also means $k \leq g(j)$.

Conversely, assume $\mathbb{A}$ is defined on an algebraic lattice whose each completely 
join-irreducible element is comparable with at least one completely 
join-irreducible element which is a fixed point of $g$. Because $\mathbb{A}$ is defined
on an algebraic lattice, there exists a set $U$ and a quasiorder $R$ on $U$ such that 
$\mathbb{A} \cong \mathbb{IRS}_R(U)$ as Nelson algebras. Let us again denote this
isomorphism by $\varphi$. 
We show that $S$ is cofinal, which by Proposition~\ref{Prop:Dense} means that
$\mathbb{IRS}_R(U)$ is an effective lattice. Let $x \in U$. If $x \in S$,
then, by reflexivity, $x \, R \, x \in S$. If $x \notin S$, then $|(R(x)| \geq 2$, and
there are two elements $j_1 <j_2$ in $\mathcal{J}$ such that $g(j_1) = j_2$,
$\varphi(j_1) = (\emptyset,\{x\}^R)$, and $\varphi(j_2) = ( R(x), R(x)^R)$. 
Because $j_1$ (or equivalently $j_2$) is  comparable with at least one completely 
join-irreducible element $k$ which is a fixed point of $g$, this
necessarily means that $j_1 < k < j_2$. Let $\varphi(k) = (\{y\},\{y\}^R)$.
It follows that $y \in S$, and now $\varphi(j_1) = (\emptyset,\{x\}^R)
\leq (\{y\},\{y\}^R) = \varphi(k)$ gives $x \in \{x\}^R \subseteq \{y\}^R$, that is,
$x \, R \, y$. Thus, $S$ is cofinal. 
\end{proof}

\section{Concluding remarks} \label{Sec:Conclusions}

The results of this paper have been suggested by the very philosophy of rough set theory.
In an indiscernibility space $(U,E)$ two elements $x,y\in U$ are indiscernible if $E(y)=E(x)=\{x,y, ...\}$.
Indeed in rough set theory, the equivalence relation $E$ of an indiscernibility space $(U, E)$ is 
induced by attributes values. Therefore, if $x$ and $y$ are indiscernible, then there is no 
property which is able to distinguish $y$ from $x$. But if $E(x)=\{x\}$, then we are given a set of 
attributes which are able to single out $x$ from the rest of the domain. In other terms, 
$x$ is uniquely determined by the set of attributes. In a sense, about $x$ we have complete information.
It is not surprise, therefore, that on the union $S$ of all the singleton equivalence classes,
Boolean logic applies. That is, $S$ is a Boolean subuniverse within a three-valued universe. The fact that
$S$ is Boolean is expressed in two ways: by saying that $X_E\cup -X^E\supseteq S$ or, equivalently, that
$X_E\cap S=X^E\cap S$, meaning that any subset $X$ is \emph{exactly defined} with respect to $S$.

When we move to quasiorders, we face the same situation. A quasiorder $R$ expresses either a preference relation
(see for instance \cite{GrecoMatarazzoSlowinski}) or an information refinement. The latter notion is embedded
in that of a {\em specialisation preorder} which characterizes Alexandrov topologies (see \cite{Vickers89}).
Thus, if $x \in S$, then $x$ is a most preferred element or a piece of information which is
maximally refined. So, $R$-closed elements decide every formula. Otherwise stated, on $S$ excluded middle is valid.

This is the logico-philosophical link between rough set systems induced by a quasiorder $R$ and effective lattices.
If $\mathbb{IRS}_R(U)$ and $\mathbb{DRS}_R(U)$ are effective lattices, then $\cong_S$ is the Glivenko congruence 
and the set $S$ is cofinal, which means that for any $x \in U$, there exist an $R$-closed element $y$ such that 
$x \, R \, y$. Indeed, this is the characteristic  which distinguishes Miglioli's Kripke models for $E_0$ 
from Thomason's Kripke models for CLSN.

At the very beginning of the paper, we have seen the reasons why Miglioli's research group introduced the
operator {\bf T} and why this operator requires that for any information state $s$
there is a complete state $s'$ which extends $s$. Here ``complete'' means that
for any atomic formula $p$, either $s'$ forces $p$ or $s'$ forces the strong negation of $p$.
Those reasons were connected to problems in program synthesis and specification.
However we can find a similar issue in other fields.

For instance, S.~Akama \cite{Akama87} considers an equivalent system endowed with modal
operators to face the "frame problem" in knowledge bases.
In that paper, intuitively, it is required that any search for
{\em complete information} must  be successfully accomplished.
On the basis of our previous discussion it is easy to understand why
Akama satisfies this request by postulating that each maximal chain
of possible worlds ends with a greatest element
fulfilling  a Boolean forcing. Hence the set of these elements is dense.

Another interesting example is given by situation theory \cite{BarwisePerry83}.
Given a situation $s$ and a state of affairs $\sigma$, $s\models\sigma$ means that situation $s$
supports $\sigma$ (or makes $\sigma$ factual). In Situation Theory some assumptions are accepted as
"natural", for any $\sigma$:
\begin{enumerate}[\rm (i)]
\item Some situation will make $\sigma$ or its dual factual:\\
$\exists s(s\models\sigma \text{ or } s\models  {\sim} \sigma$).
\item No situation will make both $\sigma$ and its dual factual:\\
$\neg\exists s(s\models\sigma \text{ and } s\models {\sim} \sigma$).
\item\label{incomplete} Some situation will leave the relevant issue unsolved
(it is admitted that for some $s$, $s\nvDash\sigma$ and $s\nvdash {\sim} \sigma$).
\end{enumerate}
In contrast with assumption \eqref{incomplete}, the following, on the contrary, is a controversial thesis:
\begin{quote}
\it There is a largest total situation which resolves all issues.
\end{quote}
It is immediate to see that this thesis is connected to the scenario depicted by logic $E_0$.

%

\begin{thebibliography}{10}

\bibitem{Akama87}
S.~Akama, \emph{Presupposition and frame problem in knowledge bases}, The Frame
  Problem in Artificial Intelligence (F.~M. Brown, ed.), Kaufmann, Los Altos,
  CA, 1987, pp.~193--203.

\bibitem{BarwisePerry83}
Jon Barwise and John Perry, \emph{Situations and attitudes}, MIT Press,
  Cambridge, MA, 1983.

\bibitem{blyth2005lattices}
Thomas~S. Blyth, \emph{Lattices and ordered algebraic structures}, Springer,
  London, 2005.

\bibitem{Cign86}
Roberto Cignoli, \emph{The class of {K}leene algebras satisfying an
  interpolation property and {N}elson algebras}, Algebra Universalis (1986),
  262--292.

\bibitem{Gabbay97}
Dov~M. Gabbay, \emph{Labelled systems}, Oxford University Press, 1997.

\bibitem{Girard93}
Jean-Yves Girard, \emph{On the unity of logic}, Annals of Pure and Applied
  Logic \textbf{59} (1993), 201--217.

\bibitem{GrecoMatarazzoSlowinski}
Salvatore Greco, Benedetto Matarazzo, and Roman S{\l}owinski, \emph{Data mining
  tasks and methods: Classification: multicriteria classification}, Handbook of
  data mining and knowledge discovery (W.~Kl\"{o}sgen and J.~\.{Z}ytkiw, eds.).

\bibitem{Iturrioz99}
Luisa Iturrioz, \emph{Rough sets and three-valued structures}, Logic at Work.
  Essays Dedicated to the Memory of Helena Rasiowa (Ewa Or{\l}owska, ed.),
  Physica-Verlag, 1999, pp.~596--603.

\bibitem{Jarv07}
Jouni J{\"a}rvinen, \emph{Lattice theory for rough sets}, Transactions on Rough
  Sets \textbf{{VI}} (2007), 400--498.

\bibitem{JarPagRad11}
Jouni J\"{a}rvinen, Piero Pagliani, and S\'{a}ndor Radeleczki, \emph{Atomic
  information completeness in generalised rough set systems}, Extended abstracs
  of the 3rd international workshop on rough set theory (RST11) 14--16
  September 2011, Milano, Italy.

\bibitem{JarRad}
Jouni J{\"a}rvinen and S\'{a}ndor Radeleczki, \emph{Representation of {N}elson
  algebras by rough sets determined by quasiorders}, Algebra Universalis
  \textbf{66} (2011), 163--179.

\bibitem{JRV09}
Jouni J{\"a}rvinen, S{\'a}ndor Radeleczki, and Laura Veres, \emph{Rough sets
  determined by quasiorders}, Order \textbf{26} (2009), 337--355.

\bibitem{KreiselPutnam}
G.~Kreisel and H.~Putnam, \emph{Eine unableitbarkeitsbeweismethode f{\"u}r den
  {I}ntuitionistischen {A}ussegenkalk{\"u}l}, Archiv f{\"u}r Mathematische
  Logik und Grundlagenferschung \textbf{3} (1957), 74--78.

\bibitem{Medvedev62}
Yu.~T. Medvedev, \emph{Finite problems}, Doklady \textbf{142} (1962),
  1015--1018, in Russian, translation in Soviet Mathematics {\bf 3}, pp.
  227--230.

\bibitem{Miglioli89b}
Pierangelo Miglioli, Ugo Moscato, Mario Ornaghi, Silvia Quazza, and Gabriele
  Usberti, \emph{Some results on intermediate constructive logics.}, Notre Dame
  Journal of Formal Logic \textbf{30}, 543--562.

\bibitem{Miglioli89a}
Pierangelo Miglioli, Ugo Moscato, Mario Ornaghi, and Gabriele Usberti, \emph{A
  constructivism based on classical truth}, Notre Dame Journal of Formal Logic
  \textbf{30} (1989), 67--90.

\bibitem{NagaUma11}
E.~Nagarajan and D.~Umadevi, \emph{A method of representing rough sets system
  determined by quasi orders}, Order, In print. DOI: 10.1007/s11083-011-9245-x.

\bibitem{Nelson49}
David Nelson, \emph{Constructible falsity}, Journal of Symbolic Logic
  \textbf{14} (1949), 16--26.

\bibitem{Pagliani90}
Piero Pagliani, \emph{Remarks on special lattices and related constructive
  logics with strong negation}, Notre Dame Journal of Formal Logic \textbf{31}
  (1990), 515--528.

\bibitem{Pagliani93}
\bysame, \emph{A pure logic-algebraic analysis of rough top and rough bottom
  equalities}, Proceedings of the International Conference on Rough Sets and
  Knowledge Discovery, Banff, 1993.

\bibitem{Pagliani96}
\bysame, \emph{Rough sets and {N}elson algebras}, Fundamenta Informaticae
  \textbf{27} (1996), 205--219.

\bibitem{Pagliani97}
\bysame, \emph{Rough set systems and logic-algebraic structures}, Incomplete
  Information: Rough Set Analysis (Ewa Or{\l}owska, ed.), Physica-Verlag, 1998,
  pp.~109--190.

\bibitem{pagliani2008geometry}
Piero Pagliani and Mihir Chakraborty, \emph{A geometry of approximation.
  {R}ough set theory, logic, algebra and topology of conceptual patterns},
  Springer, 2008.

\bibitem{Pawl82}
Zdzis{\l}aw Pawlak, \emph{Rough sets}, International Journal of Computer and
  Information Sciences \textbf{11} (1982), 341--356.

\bibitem{Rasiowa74}
Helena Rasiowa, \emph{An algebraic approach to non-classical logics},
  North-Holland, Amsterdam, 1974.

\bibitem{Sendlewski90}
Andrzej Sendlewski, \emph{Nelson algebras through {H}eyting ones~{I}}, Studia
  Logica \textbf{49} (1990), 105--126.

\bibitem{Stone68}
Arthur~Harold Stone, \emph{On partitioning ordered sets into cofinal subsets},
  Mathematika \textbf{15} (1968), 217--222.

\bibitem{Vaka77}
Dimiter Vakarelov, \emph{Notes on {N}-lattices and constructive logic with
  strong negation}, Studia Logica \textbf{36} (1977), 109--125.

\bibitem{Vickers89}
Stephen Vickers, \emph{Topology via logic}, Cambridge University Press, 1989.

\end{thebibliography}
\providecommand{\bysame}{\leavevmode\hbox to3em{\hrulefill}\thinspace}
\providecommand{\MR}{\relax\ifhmode\unskip\space\fi MR }
\providecommand{\MRhref}[2]{%
  \href{http://www.ams.org/mathscinet-getitem?mr=#1}{#2}
}
\providecommand{\href}[2]{#2}

\medskip\medskip\medskip\medskip

\end{document}